\documentclass[10pt,reqno]{amsart}
\usepackage{amsmath}
\usepackage{amssymb}
\usepackage{amsthm}
\usepackage{eepic,epic}
\usepackage{epsfig}
\usepackage{graphicx,color}
\usepackage{xcolor}

\newlength{\defbaselineskip}
\numberwithin{equation}{section}

\newtheorem{thm}{Theorem}[section]

\newtheorem{lem}[thm]{Lemma}

\theoremstyle{definition}

\theoremstyle{remark}

\numberwithin{equation}{section}

\makeatletter
\@namedef{subjclassname@2020}{\textup{2020} Mathematics Subject Classification}
\makeatother

\begin{document}

\title{Pointwise convergence for the elastic wave equation}

\author{Chu-Hee Cho, Seongyeon Kim, Yehyun Kwon and Ihyeok Seo}

\thanks{This work was supported by NRF-2020R1I1A1A01072942 (C. Cho), a KIAS Individual Grant (MG082901) at Korea Institute for Advanced Study (S. Kim), a KIAS Individual Grant (MG073701) at Korea Institute for Advanced Study and NRF-2020R1F1A1A01073520 (Y. Kwon), and NRF-2022R1A2C1011312 (I. Seo).}


\address{Department of Mathematical Sciences and RIM, Seoul National University, Seoul 08826, Republic of Korea}
\email{akilus@snu.ac.kr}

\address{School of Mathematics, Korea Institute for Advanced Study, Seoul 02455, Republic of Korea}
\email{synkim@kias.re.kr}
\email{yhkwon@kias.re.kr}

\address{Department of Mathematics, Sungkyunkwan University, Suwon 16419, Republic of Korea}
\email{ihseo@skku.edu}

\begin{abstract}
We study pointwise convergence of the solution to the elastic wave equation to the initial data which lies in the Sobolev spaces. We prove that the solution converges along every lines to the initial data almost everywhere whenever the initial regularity is greater than one half. We show this is almost optimal.
\end{abstract} 

\maketitle

\section{Introduction}
We consider the Cauchy problem for the elastic wave equation
\begin{equation}\label{e:elastic}
\begin{cases}
	\partial_t^2 u -\Delta^\ast u = 0, \quad (x,t)\in\mathbb{R}^n\times\mathbb{R}, 
	\\ u(x,0)=f(x), \quad \partial_t u(x,0)=0,
\end{cases}
\end{equation}
where $\Delta^\ast$ denotes the Lam\'e operator defined by
\begin{equation*} 
	\Delta^\ast u =\mu\Delta u +(\lambda + \mu)\nabla \mathrm{div} u.
\end{equation*}
The Laplacian $\Delta$ acts on each component of a vector field $u$.  Moreover, the following standard condition on the Lam\'e constants $\lambda,\mu\in\mathbb{R}$ is imposed to guarantee the ellipticity of $\Delta^\ast$: 
\begin{equation} \label{e:elliptic}
	\mu >0, \quad \lambda + 2\mu >0.
\end{equation}
The equation in \eqref{e:elastic} has been widely used as a model of wave propagation in an elastic medium, where $u$ denotes the displacement field of the medium (see e.g., \cite{LL,MH}). In particular, it is the classical wave equation if $\lambda + \mu = 0$.

In this paper, we are concerned with determining the optimal regularity $s>0$ such that 
\begin{equation*}
	\lim_{t\rightarrow 0} e^{it\sqrt{-\Delta^\ast}} f(x) = f(x)\ \  \text{a.e.} \ \  x
\end{equation*}
for all $f \in H^s (\mathbb{R}^n)$. The elastic wave propagator $e^{it\sqrt{-\Delta^\ast}}$ provides the solution  $u=\frac12(e^{it\sqrt{-\Delta^\ast}}f+e^{-it\sqrt{-\Delta^\ast}}f)$ for the problem \eqref{e:elastic} (see \cite{KKS}).

For the classical wave equation, the pointwise convergence has been studied by means of the maximal estimate
\begin{equation}\label{00}
\Big\|\sup_{t\in(0,1)} \big|e^{i t\sqrt{-\Delta}} f(x)\big|\Big\|_{L^2(\mathbb{R}^n)} \leq C\|f\|_{H^s(\mathbb{R}^n)}
\end{equation}
which implies the convergence.
Cowling \cite{Cow} proved \eqref{00} for $s>1/2$.
On the other hand, it was shown by Walther \cite{W} that \eqref{00} fails for $s\leq 1/2$. (See also \cite{RV} for maximal estimates in $L^p$.) While the convergence is rather well understood for the wave equation, to the best of the authors' knowledge, there doesn't seem to be any corresponding literature for the elastic wave equation. In this regard, it would be interesting to ask whether the results for the wave equation are still valid for the elastic case. We obtain the following.

\begin{thm}\label{maximal}
	Let $\theta\in S^{n-1}$ and $v \ge 0$. 
	If $f \in H^s(\mathbb{R}^n)$ with $s>1/2$, then
	\begin{equation} \label{max}
		\bigg\|\sup_{t\in(-1,1)} \Big|e^{i t\sqrt{-\Delta^\ast}} f(x+v t\theta)\Big|\bigg\|_{L^2(\mathbb{R}^n)} \leq C\|f\|_{H^s(\mathbb{R}^n)}
	\end{equation}
	uniformly in $\theta$.
Furthermore, this estimate fails if $s<1/2$.
\end{thm}

By a standard argument, Theorem \ref{maximal} implies that 
\begin{equation} \label{pwconv}
	\lim_{t \rightarrow 0}e^{i t\sqrt{-\Delta^\ast}} f(x+v t\theta) = f(x)\ \  \text{a.e.} \ \  x
\end{equation}
for $f \in H^s(\mathbb{R}^n)$ with $s>1/2$.
If $s < 1/2$, it follows from Stein's maximal principle (\cite{St}) that \eqref{pwconv} fails. The convergence \eqref{pwconv} says that the solution converges to initial data along the line $\{(x+vt\theta,t)\colon t\ge0\}$ on the light cone with speed $v$; see Figure \ref{vt}. The convergence along lines on the light cone is new even for the wave equation (the case $\lambda+\mu=0$).

\begin{figure} [t] \label{vt}
	\begin{center}
		{\includegraphics[width=0.6\textwidth]{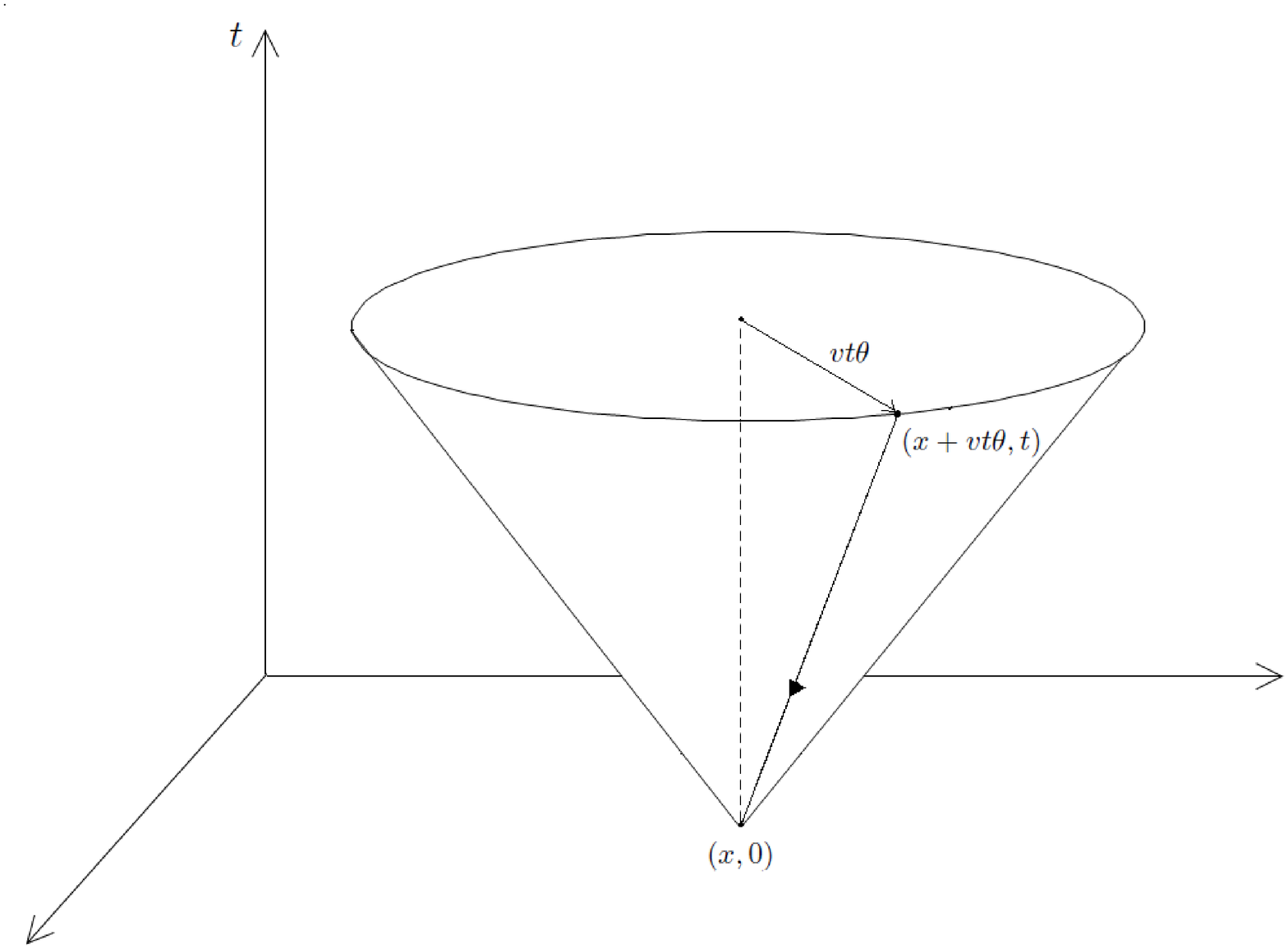}}
	\end{center}
	\caption{The light cone with speed $v$}
\end{figure}

\subsubsection*{Notations}
We denote by $S^{n-1}$ the unit sphere in $\mathbb R^n$ centered at the origin.
We use
$\|f\|_{L^2(\mathbb R^n)} = \|\{f_j\}\|_{l_j^2 L^2}$ for a vector-valued function $f=(f_1, \ldots, f_n )$, and $|M|=\big(\sum_{i,j=1}^n |M_{ij}|^2\big)^{1/2}$ for a matrix $M=(M_{ij})$.
The letter $C$ stands for a positive constant which may be different
at each occurrence.
Also, $A\lesssim B$ denotes $A\leq CB$ for a constant $C>0$, and $A\sim B$ denotes $A \lesssim B \lesssim A$.

\section{Diagonlaization of $\Delta^\ast$} \label{sec2}
Recently, the three of the authors and Lee \cite{KKLS} diagonalized the Lam\'e operator  so that $\Delta^\ast = P\Delta P^{-1}$ with a certain invertible matrix $P$  to study the elastic wave equation. We utilize the diagonalization to prove Theorem \ref{maximal} in the following sections.

Let us recall from \cite[Section 2]{KKLS} the diagonalization process and notations to make this article self-contained (the readers are encouraged to refer to \cite{KKLS} for details; also see Figure \ref{rot}). 
For the unit vector $e_1=(1,0,\ldots, 0)^t\in \mathbb R^n$, let $S_\pm =\{\omega \in S^{n-1}\colon -1/\sqrt2 \le \omega\cdot (\pm e_1)\le 1\}$ and $\mathbb R_{\pm}^n=\{\xi\in \mathbb R^n\setminus \{0\}\colon \xi/|\xi|\in S_\pm \}$. 
For every $\omega\in S_\pm\setminus \{\pm e_1\}$, we define the arc $\mathcal C_\pm(\omega)=S_\pm\cap \,\mathrm{span}\{e_1,\omega\}$, which is the intersection of $S_\pm$ and the great circle on $S^{n-1}$ passing through $e_1$ and $\omega$.
Now we take $\rho_{\pm}(\omega)\in \mathrm{SO}(n)$ so that its transpose $\rho_{\pm}(\omega)^t$ is the unique rotation mapping $\omega$ to $\pm e_1$ along the arc $\mathcal C_\pm(\omega)$ and satisfying 
$\rho_\pm(\omega)^t y=y$ whenever $y\in \mathrm{span}\{e_1,\omega\}^\perp$. When $\omega=\pm e_1$ we set $\rho_{\pm}(\omega)^t=I_n$.
We define the maps $R_{\pm}:\mathbb{R}^n_{\pm} \rightarrow \mathrm{SO}(n)$ 
and the projections $\mathcal P_\pm$ by
$$R_{\pm}(\xi)=\rho_\pm (\xi/|\xi|) \ \ \textrm{and} \ \ \widehat{\mathcal P_\pm f} (\xi) = \varphi_\pm (\xi/|\xi|) \widehat f(\xi), \quad \xi\in \mathbb R^n_\pm,$$ 
where $\{\varphi_+, \varphi_-\}$ is a smooth partition of unity on $S^{n-1}$ subordinate to the covering $\{\mathrm{int} S_+, \mathrm{int} S_-\}$.
We also set $D=-i\nabla$ and denote by $m(D)f=(m \widehat f\,)^\vee$ the Fourier multiplier defined by a bounded (matrix-valued) function $m$.

\begin{figure} [t] \label{rot}
	\begin{center}
		{\includegraphics[width=0.45\textwidth]{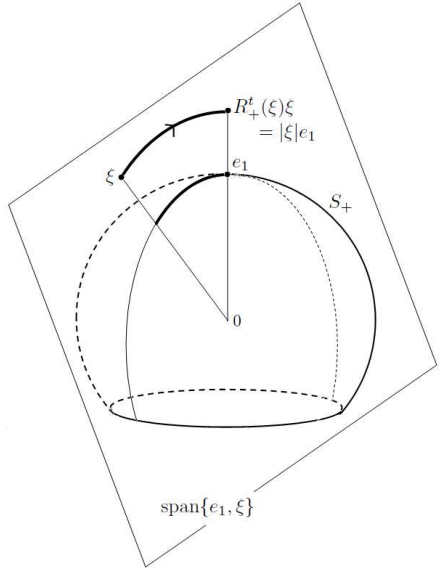}}
	\end{center}
	\caption{The rotation $R_+^t(\xi)$ for $\xi \in \mathbb{R}_+^n$}
\end{figure}

\begin{lem}[\cite{KKLS}]
	Let $L(\xi)$ be the $n\times n$ matrix-valued multiplier associated to $-\Delta^\ast$ and let $\Lambda(\xi)=\mathrm{diag} ((\lambda+2\mu)|\xi|^2, \mu|\xi|^2, \ldots, \mu|\xi|^2 )$.  
Then
\begin{equation}\label{e:diag}
	-\Delta^\ast=L(D)=\sum_\pm L(D)P_\pm =\sum_\pm R_\pm (D)\Lambda(D)R^t_\pm (D) \mathcal P_\pm.
\end{equation}
\end{lem}

The positive square root of $\Lambda$ exists by the condition \eqref{e:elliptic}, so we have  
\begin{equation}\label{e:diag2}
	\sqrt{-\Delta^\ast}=\sqrt L (D)=\sum_\pm R_\pm (D)\sqrt\Lambda (D)R^t_\pm (D)\mathcal P_\pm 
\end{equation}
with $\sqrt{\Lambda} (\xi)=\mathrm{diag} (\sqrt{\lambda+2\mu}|\xi|, \sqrt{\mu}|\xi|, \ldots, \sqrt{\mu}|\xi| )$.
Furthermore, it follows that
\begin{equation}\label{dia}
	e^{it\sqrt{-\Delta^\ast}} = \sum_\pm e^{it \sqrt L (D)}\mathcal P_\pm = \sum_\pm R_\pm(D) e^{it \sqrt \Lambda (D)} R_\pm^t(D)\mathcal P_\pm,
\end{equation}
where  
\[	e^{it \sqrt \Lambda (D)} =\mathrm{diag} \big(e^{it\sqrt{-(\lambda+2\mu)\Delta}}, e^{it\sqrt{-\mu\Delta}}, \ldots, e^{it\sqrt{-\mu\Delta}} \big).	\]

\section{Proof of Theorem \ref{maximal}: The necessity of $s\ge 1/2$}
In this section, we construct an example to show that the maximal estimate \eqref{max} implies $s\ge 1/2$. We need only to consider $\theta=e_1$. By \eqref{dia}, write
\begin{equation}\label{propa}
e^{it\sqrt{-\Delta^\ast}}f(x+vte_1) 
	= \sum_{\pm} \int_{\mathbb{R}^n} R_{\pm}(\xi) e^{ i((x+vt e_1)\cdot\xi I_n + t\sqrt{\Lambda} (\xi))} R_{\pm}^t(\xi) \widehat{\mathcal{P}_\pm f}(\xi) d\xi.
\end{equation}
We represent the $n \times n$ orthogonal matrix $R_\pm = (r^{\pm}_{ij})_{1\le i,j \le n}$ as the block matrix
$$ R_{\pm} = \begin{pmatrix}A_{\pm}
& B_{\pm} \\ C_{\pm} & D_{\pm} \end{pmatrix},$$
where $A_{\pm} = r^{\pm}_{11}$, $B_{\pm} = (r_{1j}^{\pm})_{2\le j \le n}$,  $C_{\pm}
= (r_{i1}^{\pm})_{2\le i \le n}$ and $D_{\pm}= (r_{ij}^{\pm})_{2\le i,j \le n}$.
Then
\begin{align}
	&R_{\pm}(\xi) e^{i((x+vt e_1)\cdot\xi I_n + t\sqrt{\Lambda} (\xi))} R_{\pm}(\xi)^t \nonumber \\
	&\quad = 
	\begin{pmatrix}A_{\pm}(\xi)& B_{\pm}(\xi) \\ 
	C_{\pm}(\xi) & D_{\pm}(\xi) \end{pmatrix}
	 \begin{pmatrix}e^{i\Phi(x,\xi)} & 0 \\ 
	 0  & e^{i\Psi(x,\xi)} I_{n-1} \end{pmatrix}
	 \begin{pmatrix}
	 A_{\pm}(\xi)& C_{\pm}(\xi)^t \\ 
	 B_{\pm}(\xi)^t & D_{\pm}(\xi)^t \end{pmatrix} \label{matrix} 
\end{align} 
where 
$$\Phi(x,\xi)=(x+vt e_1)\cdot \xi + t \sqrt{\lambda+2\mu}\,|\xi| \ \ \textrm{and} \ \  \Psi(x,\xi)=(x+vt e_1)\cdot \xi + t\sqrt{\mu}\,|\xi|.$$

Let us set $\alpha =\sqrt{\lambda+2\mu}+v$. For every $N \gg \alpha^{-1}$ we define
\begin{align*}
	E &=\big \{x \in \mathbb{R}^n \colon \angle(e_1,x)\le 2^{-2}(\alpha N)^{-\frac 12} ,\,  |x| \leq \alpha \big\}, \\
	F &=\big\{\xi \in \mathbb{R}^n \colon \angle(e_1, \xi)\le 2^{-2}(\alpha N)^{-\frac 12} ,\,  N/2 \le |\xi| \le N \big\},
\end{align*} 
where $\angle(e_1, x)$ denotes the angle between $e_1$ and $x$. It is easy to see that
\[	|F| \sim N^{\frac{n+1}2} \ \ {\rm{and}} \ \ |E| \sim N^{-\frac{n-1}2}.	\]

If we take 
\begin{equation} \label{t}
	t=t(x) = \frac{-|x|}\alpha\in (-1,0),
\end{equation}
then
\begin{align*}
	|\Phi(x,\xi)|&= \big|(x+vt(x)e_1)\cdot \xi + t(x)\sqrt{\lambda+2\mu}\,|\xi| \big| \\
	&\le |x||\xi|\Big( \big(1-\cos\angle(x,\xi)) + \frac{v}{\alpha} \big(1-\cos \angle(e_1,\xi) \big)\Big) .
\end{align*}
Since $\cos u \ge 1- u^2/2$ we have for $(x,\xi) \in E\times F$ 
\[	1-\cos\angle(x,\xi) \le 2^{-3} (\alpha N)^{-1} \  \  \text{and}  \  \
	1-\cos\angle(e_1,\xi) \le 2^{-5} (\alpha N)^{-1},	\]
from which it follows that
\begin{equation}
	|\Phi(x,\xi)| \leq \alpha N \Big(2^{-3}(\alpha N)^{-1} + \frac{v}{\alpha}2^{-5}(\alpha N)^{-1} \Big) \leq 2^{-2} \label{phi}
\end{equation}
since $v/\alpha<1$.

Let us define $f_F\colon \mathbb R^n\to \mathbb R^n$ by $\widehat {f_F}(\xi)= \chi_F(\xi) e_1 $. By the definition of $F$ and $\mathcal P_\pm$ it is clear that $\mathcal P_+ f_F =f_F$ and $\mathcal P_-f_F= 0$. By \eqref{propa} and \eqref{matrix}, we have 
\begin{align}
	&\sup_{t\in(-1,1)} \big| e^{i t\sqrt{-\Delta^\ast}} f_F(x+v te_1)\big| \nonumber \\
	&\qquad \ge \sup_{t=t(x)} \bigg| \int_F \begin{pmatrix} e^{i\Phi (x,\xi)} A_+(\xi)^2  + e^{i\Psi(x,\xi)} B_+(\xi) B_+(\xi)^t  \\  e^{i\Phi(x,\xi)} C_+(\xi) A_+(\xi)  + e^{i\Psi(x,\xi)} D_+(\xi) B_+(\xi)^t \end{pmatrix} d\xi \bigg|  \nonumber \\
	&\qquad \ge \sup_{t=t(x)} \bigg| \int_F  e^{i\Phi (x,\xi)} A_+(\xi)^2  + e^{i\Psi(x,\xi)} B_+(\xi) B_+(\xi)^t   d\xi \bigg|  \nonumber \\
	&\qquad \ge  \sup_{t=t(x)} \int_{F}  \cos\Phi(x,\xi)  A_+(\xi)^2 -  B_+(\xi) B_+(\xi)^t d\xi . \label{e:lower}
\end{align}
We need to estimate the size of $A_+(\xi)$ and $B_+(\xi)$ for $\xi\in F$.  
\begin{lem}
If $\xi\in F$ then 
\begin{equation}\label{e:blocksize}
|1-A_+(\xi)|\lesssim N^{-\frac12} \ \ \text{and} \ \ |B_+(\xi)| \lesssim  N^{-\frac12}. 
\end{equation}
\end{lem}
\begin{proof}
Let us write $\omega=\xi/|\xi|$ for $\xi\in F$ and recall from Section \ref{sec2} that  $\rho_+(\omega)^t \omega= e_1$ and $\rho_+(e_1)=I_n$. We observe that 
\begin{align*}
	( \rho_+(\omega)^t -I_n ) e_1 
	&= ( \rho_+(\omega)^t -I_n ) \omega + ( \rho_+(\omega)^t -I_n ) (e_1-\omega)\\
	&= e_1-\omega + ( \rho_+(\omega)^t -I_n ) (e_1-\omega) \\
	&= \rho_+(\omega)^t(e_1-\omega) ,
\end{align*}
which leads us to 
\[	|(A_+(\xi) -1, B_+(\xi))|=|( \rho_+(\omega)^t -I_n ) e_1 | =|e_1-\omega| \le \angle(e_1, \xi) \lesssim (\alpha N)^{-\frac{1}{2}} .	\]
This completes the proof \eqref{e:blocksize}.
\end{proof}
By the estimates \eqref{e:blocksize}, we have 
\begin{equation}\label{e:blocksize1}
A_+(\xi)^2 \gtrsim 1-N^{-1} \ \ \text{and} \ \  B_{+}(\xi) B_{+}(\xi)^t \lesssim N^{-1}.
\end{equation}
Combining these with \eqref{phi} we see that \eqref{e:lower} is estimated by
\[	  	 
\gtrsim (1-N^{-1}) \sup_{t=t(x)} \int_{F} \cos\Phi(x,\xi) d\xi -  \int_{F} N^{-1} d\xi  \gtrsim |F|.	\]
for every $N$ large enough. Thus, we obtain
\begin{equation}\label{left}
	\Big\|\sup_{t\in(-1,1)} \big|e^{i t\sqrt{-\Delta^\ast}} f_F(x+v t e_1)\big|\Big\|_{L^2(\mathbb{R}^n)} 
	\gtrsim |F||E|^{\frac12}.
\end{equation}
On the other hand,
\begin{equation}\label{right}
	\|f_F\|_{H^s(\mathbb{R}^n)} = \Big(\int_{\mathbb{R}^n} (1+|\xi|^2)^{s} \widehat{f_F}(\xi) d\xi \Big)^\frac12 \lesssim |F|^{\frac12}N^s,
\end{equation}
By  \eqref{left} and \eqref{right}, the maximal estimate \eqref{max} implies 
\begin{equation} \label{max1}
	 N^s \gtrsim |F|^\frac12 |E|^\frac12  \sim N^{\frac12}
\end{equation}
for all $N$ large enough. Therefore we conclude that the estimate \eqref{max} does not hold for $s< 1/2$.

\section{Proof of the estimate \eqref{max}} 
Let us fix a cutoff function $\phi\in C_0^{\infty}(\mathbb{R})$ such that $\phi(t)=1$ for $|t|\le 1$ and $\phi(t)=0$ for $|t|\ge 2$, and denote
\begin{equation*}
T_{v,\theta}f(x,t)= \phi(t)e^{i t\sqrt{-\Delta^\ast}}f(x+v t \theta).
\end{equation*}
In order to prove \eqref{max} we need only to show the following:
\begin{align}
	\label{xt}
	\|T_{v,\theta}f\|_{L^2_{x,t}(\mathbb{R}^{n+1})} &\lesssim \|f\|_{L^2(\mathbb{R}^n)}, \\
	\label{xt'}
	\|\partial_t T_{v,\theta}f\|_{L^2_{x,t}(\mathbb{R}^{n+1} )} &\lesssim \|f\|_{H^1(\mathbb{R}^n)}.
\end{align}
Indeed, the estimates imply 
\begin{equation}\label{t1}
	\|T_{v,\theta}f\|_{L_x^2(\mathbb{R}^n;H_t^1(\mathbb{R}))} \lesssim \|f\|_{H^1(\mathbb{R}^n)}.
\end{equation}
Furthermore, by complex interpolation (see \cite{BL})
it follows from \eqref{xt} and \eqref{t1} that, for $0<s<1$,
\begin{equation}\label{int}
	\|T_{v,\theta}f\|_{L_x^2(\mathbb{R}^n;H_t^s(\mathbb{R}))} \lesssim \|f\|_{H^s(\mathbb{R}^n)}.
\end{equation}
By the Sobolev embedding $H^s(\mathbb{R}) \hookrightarrow L^\infty(\mathbb{R})$ for $s>1/2$, we conclude 
\begin{equation*}
	\|T_{v,\theta}f\|_{L_x^2 (\mathbb{R}^n;L_t^\infty(\mathbb{R}))} 
	\lesssim \|f\|_{H^s(\mathbb{R}^n)},
\end{equation*}
which completes the proof of  \eqref{max}.

Now let us show \eqref{xt} and \eqref{xt'}.
Using the diagonalization \eqref{dia} we write
\begin{equation} \label{e:tvth}
	T_{v,\theta}f(x,t)
	= \phi(t) \sum_\pm  \int_{\mathbb{R}^n} e^{ ix\cdot\xi I_n} R_\pm(\xi)  e^{it(v \theta \cdot \xi I_n +  \sqrt \Lambda (\xi))} R_\pm(\xi)^t \widehat{\mathcal P_\pm f}(\xi) d\xi.
\end{equation}
Then, by the Plancherel theorem 
\begin{align*}
	\|T_{v,\theta}f\|_{L^2_{x,t}(\mathbb{R}^{n+1})} ^2
	&\lesssim \sum_\pm  \int_{\mathbb{R}^{n+1}}  \big| \phi(t)  R_\pm(\xi)  e^{it(v \theta \cdot \xi I_n +  \sqrt \Lambda (\xi))} R_\pm(\xi)^t \widehat{\mathcal P_\pm f}(\xi) \big|^2 d\xi dt \\
	&\lesssim\|f\|_{L^2(\mathbb{R}^n)} ^2.
\end{align*}
On the other hand, $|\partial_t T_{v,\theta}f(x,t)|$ is estimated by
\begin{align*}
&\bigg| \phi' (t) \sum_\pm  \int_{\mathbb{R}^n} e^{ ix\cdot\xi I_n} R_\pm(\xi)  e^{it(v \theta \cdot \xi I_n +  \sqrt \Lambda (\xi))} R_\pm(\xi)^t \widehat{\mathcal P_\pm f}(\xi) d\xi \bigg| \\
&+ \bigg| \phi(t) \sum_\pm  \int_{\mathbb{R}^n} e^{ ix\cdot\xi I_n} R_\pm(\xi) (v\theta \cdot \xi I_n +\sqrt{\Lambda})  e^{it(v \theta \cdot \xi I_n +  \sqrt \Lambda (\xi))} R_\pm(\xi)^t \widehat{\mathcal P_\pm f}(\xi) d\xi \bigg|.
\end{align*}
Then by the Plancherel theorem, we obtain \eqref{xt'}.


\end{document}